\documentclass[12pt,a4paper,bibliography=totoc,abstract=on,twoside]{scrreprt}
\usepackage[english]{babel}
\usepackage{amsmath}
\usepackage{amsfonts}
\usepackage{amssymb}
\usepackage{amsthm}
\usepackage{graphicx}
\usepackage{wasysym}
\usepackage{listings}
\usepackage{setspace}
\usepackage{mathrsfs}
\usepackage[titles]{tocloft}
\usepackage{tikz}
\usepackage[a4paper,inner=3.5cm,outer=3cm,top=3cm,bottom=4cm,pdftex]{geometry}

\setkomafont{sectioning}{\rmfamily\bfseries} 

\theoremstyle{plain}
\newtheorem{thm}{Theorem}[subsection]
\newtheorem{lem}[thm]{Lemma}
\newtheorem{cor}[thm]{Corollary}
\newtheorem{prop}[thm]{Proposition}

\theoremstyle{remark}
\newtheorem{rem}[thm]{Remark}
\newtheorem{note}[thm]{Note}

\theoremstyle{definition}
\newtheorem{defn}[thm]{Definition}

\newtheorem{exmp}[thm]{Example}

\allowdisplaybreaks
\begin{document}
\pagenumbering{arabic}

\title{Model theoretic construction for layered semifields}
\author{Tal Perri}%
\maketitle

\makeatletter
\let\thetitle\@title
\let\theauthor\@author
\makeatother

\begin{abstract}
In this paper we introduce a model theoretic construction for the theories of uniform layered domains and semifields
introduced in the paper of Izhakian, Knebusch and Rowen (\cite{Layered}).
We prove that, for a given layering semiring $L$, the theory of uniform $L$-layered divisibly closed semifields is complete. In the process of doing so, we prove that this theory has quantifier elimination and consequently is model complete.
Model completeness of uniform $L$-layered divisibly closed has some important consequences regarding the uniform $L$-layered semifields theory. One example involves equating polynomials. Namely, model completeness insures us that if two polynomials are equal over a divisibly closed uniform $L$-layered semifield, then they are equal over any divisibly closed uniform $L$-layered extension of that semifield, and thus over any uniform $L$-layered domain extending the semifield (as it is contained in its divisible closure of its semifield of fractions).
At the end of this paper we apply our results to the theory of max-plus algebras as a special case of uniform $L$ -layered domains.
\end{abstract}

\section{Overview}

In this paper, we introduce a model theoretic construction for the theories of \linebreak uniform layered domains and semifields and prove that for a given layering semifield $L$, the theory of uniform $L$-layered divisibly closed semifields is complete.
In the first section we introduce the model theoretic construction. This will be done by means of projections onto the set of 'values' and the set of 'layers' of the uniform layered domain algebraic structure. For a given algebraic structure, there might exist several constructions yielding the exact same theory. From a model theoretic perspective,
one aims for a construction that provides a quick tool  to work with. For instance, the more relations and existence axioms introduced in a construction, the more complicated is the task of deriving results about the theory. The importance of the construction we introduce is that it is fairly simple to work with.
In the subsequent sections we use our construction to prove that for given layering semiring $L$, the theory of divisibly closed uniform $L$-layered semifields to be complete. In particular, this provides a quick, direct proof of \cite{Layered}. We start by subsequently characterizing the building blocks of the theory: terms, atomic formulas and general formulas (section 2), and then derive our desired result (section 3).

\bigbreak

Before we start, we give some preliminaries:\\

\begin{note}
\textbf{1.} \ In this paper we make extensive use of the ideas developed in Layered Tropical Mathematics paper (\cite{Layered}).\\
\textbf{2.} \ The book of David Marker \cite{Mark} is used as a standard reference for all model theory
definitions and ideas appearing in this paper .\\
\textbf{3.} \ There is a similarity of notation in super-tropical theory and model theory which
may cause some confusion. Namely, while in model theory, for a structure $A$ and a theory $T$,
$A \models T$ denotes the relation $A$  \ `models' $T$, it also denotes the `surpasses' relation
in the super-tropical theory. To avoid confusion, we note that in this paper, $\models_L$ denotes the $L$-surpasses
relation of the theory of layered domains. We use the notation $\models$ solely to denote
the model theoretic `models' relation.

\end{note}

\begin{rem}
We use the symbol $R^{\dag}$ to indicate a semiring without $0$.
We use the notion of `cancellative' for a (multiplicative) semigroup $G$, to indicate that
for any $a,b,c \in G$ such that $c \neq 0$, if $ac = bc$ then $a = b$.
We use the notion of `divisibly closed' for a layered structure in the sense of $1$-divisibly closed layered structure
as defined in \cite{Layered}.
\end{rem}

\begin{rem}\label{generic}
The generic uniform $L$-layered domain is an algebraic structure defined as
$\mathcal{R}(L, \mathcal{G}) = \{ ^{[l]}a \ | \ a \in \mathcal{G}, \ l \in L \}$
where $L$ is a cancellative semiring$^{\dag}$ and $\mathcal{G}$ is a cancellative ordered monoid viewed as
a semiring$^{\dag}$ in which addition is given by $a+b = \max \{a,b \}$ with respect to the order
of $\mathcal{G}$. If $\mathcal{G}$ is taken to be a semifield, $\mathcal{R}(L, \mathcal{G})$ is said to be a
uniform $L$-layered $1$-semifield.\\
These structures are defined in detail in the paper \cite{Layered}.
\end{rem}
%

\begin{note}
We use the term \emph{domain} for a  cancellative ordered monoid $M$ viewed as
a semiring, in which addition is given by $a+b = \max \{a,b \}$ with respect to the \linebreak order
of $M$.
\end{note}
%

\newpage

\section{Model theory of uniform layered domains} \ \\

We start by constructing a theory for uniform layered domains and divisibly closed uniform $1$-semifields,
as first introduced in \cite{Layered}.\\

\begin{note}
Throughout the rest of this paper, as we deal only with uniform layered domains, we usually omit the use of the word `uniform'.
\end{note}

The main idea of our construction is to view a layered domain as a set theoretic cartesian product of two sets.
The first is the `fiber' of $\mathcal{G}$ values obtained by the \linebreak restriction to the unit layer $1_{L}$ of $L$, while the second is the
layers `fiber' obtained by the restriction to the unit $1_{G}$ of $\mathcal{G}$. The operations of the layered domain
induce (essentially different) operations on these subsets to form a pair of semiring structures. The nice things about these resulting
semirings is that they can be used to reconstruct the original uniform domain. We use this property by defining a pair of projections,
$\pi_1$ and $\pi_2$, one for each of these special semirings. Through these projections, we introduce the axioms defining the uniform layered
domain, by defining them on these special sets. The main advantage in this construction is that it keeps most of the axioms
in universal form which in turn, are much easier to work with in model theory.

\medbreak
Before we start our construction and throughout the rest of this paper, we will \linebreak introduce relevant model theoretic concepts before using them.\\
\bigbreak
\begin{defn}
A \emph{language} $\mathcal{L}$ is given by specifying the following data:
\begin{enumerate}
  \item A set of function symbols $\mathcal{F}$ and positive integers $n_{f}$ for each $f \in \mathcal{F}$.
  \item A set of relation symbols $\mathcal{R}$ and positive integers $n_{R}$ for each $R \in \mathcal{R}$.
  \item A set of constant symbols $\mathcal{C}$.
\end{enumerate}
\end{defn}

\begin{defn}
Let $\mathcal{L}$ be a language. An $\mathcal{L}$-\emph{theory} $T$ is a set of $\mathcal{L}$-sentences.
\end{defn}

\begin{defn} \label{theory_defn}
Let $\mathcal{L}^{\star} = \{+ , \cdot   ,  \pi_{1}, \pi_{2} , < , 1, 0  \}$
where $\cdot$ and $+$  are binary function symbols, $\pi_{1}, \pi_{2}$ are unary functions symbols and
$1, 0$  are constant symbols.

We define the $\mathcal{L}^{\star}$-theory $\mathrm{T}$ satisfying the following axioms:

$$\forall x \ 1 \cdot x = x \cdot 1 = x$$
$$\forall x \forall y \  x \cdot y = y \cdot x$$
$$\forall x \forall y \forall z \  x \cdot ( y \cdot z ) = ( x \cdot y ) \cdot z$$
$$\forall x  (x \neq 0 \rightarrow \ \exists y  \ y \cdot \pi_1(x) = 1)$$
$$\forall x \forall y ((x = 0) \vee (y = 0) \leftrightarrow x \cdot y = 0)$$

$$\forall x \  0 + x = x + 0 = x$$
$$\forall x \ 0 \cdot x = x \cdot 0 = 0$$

$$\forall x \forall y  \  \pi_{1}(x) + \pi_{1}(y) = \pi_{1}(y) + \pi_{1}(x)$$
$$\forall x \forall y \forall z  \ \pi_{1}(x) + ( \pi_{1}(y) + \pi_{1}(z) ) = ( \pi_{1}(x) + \pi_{1}(y) ) + \pi_{1}(z)$$

$$\forall x \forall y  \  \pi_{2}(x) + \pi_{2}(y) = \pi_{2}(y) + \pi_{2}(x)$$
$$\forall x \forall y \forall z  \ \pi_{2}(x) + ( \pi_{2}(y) + \pi_{2}(z) ) = ( \pi_{2}(x) + \pi_{2}(y) ) + \pi_{2}(z)$$

$$\pi_{1}(1) = \pi_{2}(1) = 1$$
$$\pi_{1}(0) = 0$$
$$\pi_{2}(0) = 1$$

\begin{equation}\label{nontriviality}
\exists x \ ((\pi_1(x) \neq 0) \wedge (\pi_1(x) \neq 1))
\end{equation}



$$\forall x \forall y \ (x < y ) \leftrightarrow (\pi_{1}(x) < \pi_{1}(y))$$
$$\forall x \forall y \forall z \ (\pi_{1}(x) < \pi_{1}(y))  \wedge (\pi_{1}(y) < \pi_{1}(z))  \rightarrow \pi_{1}(x)< \pi_{1}(z)  $$
$$\forall x \forall y \ ( \pi_{1}(x) < \pi_{1}(y)  \vee \pi_{1}(x) = \pi_{1}(y)  \vee \pi_{1}(y) < \pi_{1}(x))$$

$$\forall x \ \neg(\pi_{2}(x) <  \pi_{2}(x))$$
$$\forall x \forall y \forall z \ (\pi_{2}(x) < \pi_{2}(y))  \wedge (\pi_{2}(y) < \pi_{2}(z))  \rightarrow \pi_{2}(x)< \pi_{2}(z)  $$
$$\forall x \forall y \ ( \pi_{2}(x) < \pi_{2}(y)  \vee \pi_{2}(x) = \pi_{2}(y)  \vee \pi_{2}(y) < \pi_{2}(x))$$


\begin{align}\label{sentence}
\forall x \forall y \ \pi_{1}(x \cdot y) = \pi_{1}(x) \cdot \pi_{1}(y)\end{align}
\begin{align*}\forall x \forall y \ \pi_{2}(x \cdot y) = \pi_{2}(x) \cdot \pi_{2}(y)\end{align*}

\bigbreak

\begin{align*}\forall x \forall y \ ((\pi_{1}(x) \geq \pi_{1}(y))\wedge(\pi_{1}(x + y) = \pi_{1}(x))) \vee
((\pi_{1}(x) \leq \pi_{1}(y))\wedge(\pi_{1}(x + y) = \pi_{1}(y)))\end{align*}
\begin{align*}\forall x \forall y \ & ((\pi_{1}(x) > \pi_{1}(y))\wedge(\pi_{2}(x + y) = \pi_{2}(x))) \vee
((\pi_{1}(x) < \pi_{1}(y))\wedge(\pi_{2}(x + y) = \pi_{2}(y)))\\
\ \vee & ((\pi_{1}(x) = \pi_{1}(y))\wedge(\pi_{2}(x + y) = \pi_{2}(x)+\pi_{2}(y))) \end{align*}

$$ \forall x \ \pi_{1}(\pi_{1}(x)) = \pi_{1}(x)$$
$$ \forall x \ \pi_{2}(\pi_{2}(x)) = \pi_{2}(x)$$
$$ \forall x \ \pi_{1}(\pi_{2}(x)) = \pi_{2}(\pi_{1}(x)) = 1$$

\begin{equation}\label{equalityAxiom}
\forall x \forall y \ ( x = y  \leftrightarrow (\pi_{1}(x)=\pi_{1}(y)) \wedge (\pi_{2}(x)=\pi_{2}(y))
\end{equation}

\begin{equation}\label{dclosedAxiom}
\text{For each} \  n \in \mathbb{N} : \forall x \ x = \pi_{1}(x) \rightarrow \exists y \ ((y = \pi_{1}(y)) \wedge (y^n = x))
\end{equation}
\end{defn}

\bigbreak
\bigbreak

\begin{note}
The use of the unary function symbols allows us to model the different \linebreak behavior of the value monoid
and layering (cancellative) semiring$^{\dag}$ with respect to the binary functions $\cdot$ and $+$ while keeping the formulation of the theory free from existence axioms.
\end{note}

\begin{defn}
$\mathrm{T}$ just given is called the theory of $\mathbb{N}$-divisibly closed $1$-semifields $DLSF$.
\end{defn}

\begin{flushleft}In what follows we refer to $\mathbb{N}$-divisibly closed simply as divisibly closed.\end{flushleft}
\bigbreak

\begin{rem}
A few remarks concerning the DLSF theory:\\
\begin{enumerate}
  \item Axiom \eqref{nontriviality} ensures that the `value' semiring (corresponding to $\mathcal{G}$ of Remark~\eqref{generic}) is non-trivial.

  \item One consequence of the above theory is that
$$\forall x \ (x=\pi_{1}(x) \cdot \pi_{2}(x)).$$
Indeed, $\pi_{1}(\pi_{1}(x) \cdot \pi_{2}(x))=\pi_{1}(x) \cdot 1 = \pi_{1}(x)$ and
$\pi_{2}(\pi_{1}(x) \cdot \pi_{2}(x))= 1 \cdot \pi_{2}(x) = \pi_{2}(x)$. Thus by axiom \eqref{equalityAxiom} we have that $x=\pi_{1}(x) \cdot \pi_{2}(x)$.

   \item It can be checked that all order relations are consequences of the above theory. Namely,
$$\forall x \ \neg(x < x)$$
$$\forall x \forall y \forall z \ ((x < y  \wedge y < z)  \rightarrow x< z ) $$
$$\forall x \forall y \ ( x < y  \vee x = y  \vee y < x)$$

  \item We have introduced a theory for uniform $L$-layered $1$-semifields $\mathcal{R}(L, \mathcal{G})$
where $(L,\cdot)$ need not be a semifield. In case we desire $L$ to be a semifield, and thus $\mathcal{R}(L, \mathcal{G})$ will also be a semifield,
we need to replace the axiom $$\forall x  (x \neq 0 \rightarrow \ \exists y  \ y \cdot \pi_1(x) = 1)$$
by the axiom $$\forall x (x \neq 0 \rightarrow \ \exists y  \ y \cdot x = 1).$$

  \item The last axioms \eqref{dclosedAxiom} ensure that  $\mathcal{G}$ is $\mathbb{N}$ - divisibly closed.
Notice that all axioms but these last ones along with the existence of an inverse to
$\pi_1(x)$ and the non-triviality of the image of $\pi_1$, are universal (i.e., `for all' sentences).
\end{enumerate}

\end{rem}

\bigbreak

\begin{defn}
Let $\mathcal{F}, \mathcal{R}, \mathcal{C}$ be the sets of functions, relations and constants of a language $\mathcal{L}$.
An $\mathcal{L}$-\emph{structure} $\mathcal{M}$ is given by the following data:
\begin{enumerate}
  \item A nonempty set $M$ called the \emph{universe} of $\mathcal{M}$.
  \item A function $f^{\mathcal{M}}: M^{n(f)} \rightarrow M$ for each $f \in \mathcal{F}$.
  \item A set $R^{\mathcal{M}} \subset M^{n(R)}$ for each $R \in \mathcal{R}$.
  \item An element $c^{\mathcal{M}} \in M$ for each $c \in \mathcal{C}$.
\end{enumerate}
\end{defn}

\begin{defn}
Let $\mathcal{M}$ and $\mathcal{N}$ be $\mathcal{L}$-structures with universes $M$ and $N$, respectively. An $\mathcal{L}$-\emph{embedding}, $\eta : \mathcal{M} \rightarrow \mathcal{N}$ is a one to one map $\eta : M \rightarrow N$ that preserves the interpretation of all symbols of $\mathcal{L}$. In the context of model theoretic assertions, when $\mathcal{L}$ is understood, we simply write `embedding'.
\end{defn}

\begin{defn}
Let $\mathcal{M}$ be an $\mathcal{L}$-structure, for some language $\mathcal{L}$. Then $Th(\mathcal{M})$, the \emph{full theory of} $\mathcal{M}$, is defined to be the set of all $\mathcal{L}$-sentences $\phi$ such that $\mathcal{M} \models \phi$ (i.e., all $\mathcal{L}$-sentences satisfied by $\mathcal{M}$).
\end{defn}
\begin{rem}
$Th(\mathcal{M})$ is a complete theory, i.e., for every $\mathcal{L}$-sentence $\phi$, \ either \linebreak $Th(\mathcal{M}) \models \phi$ or $Th(\mathcal{M}) \models \neg \phi$.
\end{rem}

\begin{defn}
We denote the theory of semirings by $\mathcal{L}_{sr}$, and the theory of (linearly) ordered semirings
by $\mathcal{L}_{osr}$. Both of these theories consist of universal ($\forall$) axioms.
\end{defn}

\bigbreak

We now formally define the theory of $\mathbb{N}$-divisibly closed $L$-layered $1$-semifields $DLSF(L)$ to be the $DLSF$ endowed with the extra condition that $Im(\pi_2) \subseteq L$ as ordered semirings.

\begin{defn}
Let $L$ be a given ordered semiring$^{\dag}$. Let $\mathcal{L}_{L}^{\star}$ be the language $\mathcal{L}^{\star}$ where we add to $\mathcal{L}^{\star}$ the set of constant symbols $L^{\prime} = \{ \ell^{\prime} \ : \ \ell \in L \}$ , and a single predicate (unary relation symbol) $P_{L}$ identifying the elements of $L^{\prime}$ . Define the following set of $\mathcal{L}^{\star}_{L}$ formulas:
\begin{align}\label{restriction_axioms1}
\Phi(L) & = \{ \phi(\ell_1^{\prime},...,\ell_n^{\prime}) \ : \ \phi \in Th(L) \}\\
&  = \{ \phi(\ell_1^{\prime},...,\ell_n^{\prime}) \ : \ L \models \phi(\ell_1,...,\ell_n), \ \phi \ \text{is an} \ \mathcal{L}_{osr}-\text{formula} \}.  \nonumber
\end{align}
\begin{equation}\label{restriction_axioms2}
\text{For each} \ \ell^{\prime} \in L^{\prime} \ : \ \ell^{\prime} = \pi_2(\ell^{\prime}).
\end{equation}
\begin{equation}\label{restriction_axioms3}
\forall x \ (x = \pi_2(x) \leftrightarrow P_{L}(x)).
\end{equation}
\begin{equation}\label{restriction_axioms4}
\pi_2(\ell^{\prime})=1 \leftrightarrow  \ell^{\prime} = 1.
\end{equation}

\bigbreak

\begin{flushleft} Now, define the $\mathcal{L}^{\star}_{L}$-theory $DLSF(L)$ to be the $DLSF$ theory with the above sentences \eqref{restriction_axioms1} - \eqref{restriction_axioms4} added to it. \end{flushleft}

\end{defn}

\newpage

\begin{rem}
\begin{enumerate}
  \item The symbols of $L^{\prime}$ and the sentences in and \eqref{restriction_axioms2} ensure that a copy of $L$, $L^{\prime}$, is contained in $Im(\pi_2)$. The sentences in \eqref{restriction_axioms3} ensure that this copy is  $Im(\pi_2)$ (as sets).  The formulas in \eqref{restriction_axioms1} endow $L^{\prime}$ with the algebraic structure of $L$. Note that the latter is well-defined, as $Im(\pi_2)$ is itself an ordered semiring$^{\dag}$.
 \item We have added axiom \eqref{restriction_axioms4}, since the symbol $1$ of $\mathcal{L}^{\star}$ and one of the symbols in $L^{\prime}$ must be interpreted by the same element for any model of $DLSF(L)$.
 \item All added sentences are quantifier free $\mathcal{L}^{\star}_{L}$-formulas. In particular the sentences in \eqref{restriction_axioms1} are quantifier and variable free $\mathcal{L}^{\star}_{L}$-formulas, as the $\ell_i^{\prime}$'s are constant symbols.

  \item The sentences in \eqref{restriction_axioms2} imply that $\pi_1(\ell^{\prime}) = 1$ for every $\ell^{\prime} \in L^{\prime}$.
\end{enumerate}
\end{rem}

\bigbreak

\begin{rem}
The $\mathcal{L}^{\star}$-theory of layered semidomains, $LD$, is just the theory of $DLSF$ without
the axioms of invertibility, triviality of $\pi_1$ and $\mathbb{N}$ - divisibility, namely, omitting the axioms:

$$\exists x \ ((\pi_1(x) \neq 0) \wedge (\pi_1(x) \neq 1));$$

\begin{equation*}
\text{for each} \  n \in \mathbb{N} : \forall x \ x = \pi_{1}(x) \rightarrow \exists y \ ((y = \pi_{1}(y)) \wedge (y^n = x))
\end{equation*}

$$\forall x  (x \neq 0 \rightarrow \ \exists y  \ y \cdot \pi_1(x) = 1).$$

\bigbreak

\begin{flushleft}We analogously define $LD(L)$ to be $DLSF(L)$ without the above axioms.\end{flushleft}
\end{rem}

\bigbreak

\medbreak

Uniform layered domains and uniform layered $1$-semifields are $\mathcal{L}^{\star}$-structures.
Let us denote, for simplicity, the general element $^{[l]}a$ where $l$ is an element
of the layering domain$^{\dag}$, $L$, and $a \in \mathcal{G}$, by the pair $(l,a)$.
Indeed, we interpret $0$ by $(1,-\infty)$ and $1$ by $(1,0)$. As for the function symbols, we
interpret $\cdot$ and $+$ by $\cdot_{L}$, $+_{L}$, respectively.
We interpret $\pi_{1}$ by the `evaluation' $\pi_{1}(l,a) = (1_{L},a)$ where $1_{L} \in L$ is the identity
element with respect to $\cdot_{L}$ and $\pi_{2}$ by the layering map $\pi_{2}(l,a) = (l,1_{G})$
where $1_{G} \in G$ is the identity element with respect to $\cdot_{G}$ (plus).
Finally, we interpret $<$ by $<_{L}$. \\

Analogously, given an ordered domain$^{\dag}$ , uniform $L$-layered domains and uniform \linebreak $L$-layered $1$-semifields are $\mathcal{L}^{\star}_{L}$-structures, where the elements $(\ell,1)$ \ with $\ell \in L$, interpret the constant symbols $\ell^{\prime} \in L^{\prime}$ and the set $\{ (\ell,1)  \ : \ \ell \in L \}$ interprets the unary relation symbol $P_{L}$.

As an example we note that the ghost surpasses relation introduced in \cite{Layered} can be defined as follows:

\begin{exmp}
For a uniform $L$-layered domain $M$, by definition we have the following:
for $a,b \in M$,
$$a \cong_{\nu} b \Leftrightarrow \pi_{1}(a) = \pi_{1}(b)$$ and
$$a \models_{L} b \Leftrightarrow (a=b) \vee ((\pi_{2}(a) > \pi_{2}(b))\wedge
(\pi_{1}(a) = \pi_{1}(b))) \vee ((a = b + c)\wedge (\pi_{2}(c) \geq \pi_{2}(b))).$$
\end{exmp}

\medbreak

\begin{defn}
Given an $\mathcal{L}$-theory, $\mathbb{T}$, an $\mathcal{L}$-structure $\mathcal{M}$ is a \emph{model} for
$\mathbb{T}$, written as $\mathcal{M} \models \mathbb{T}$, if $\mathcal{M} \models \phi$ ($\mathcal{M}$ satisfies  $\phi$) for all sentences(axioms) $\phi \in \mathbb{T}$.
\end{defn}

By the specifications of the axioms of layered $1$-semifield and $L$-layered $1$-semifield, it is a straightforward consequence that the $\mathcal{L}^{\star}$-structure of a layered domain and $\mathcal{L}^{\star}_{L}$-structure of an $L$-layered domain are models for $DL$ and $DL(L)$, respectively, while the $\mathcal{L}^{\star}$-structure
of a layered $1$-semifield and the $\mathcal{L}^{\star}_{L}$-structure of an $L$-layered $1$-semifield are models of $DLSF$ and $DLSF(L)$, as described in \cite{Layered}.

\newpage

\section{Building up the theories of layered and $L$-layered \\ \ \ \ \ divisibly closed $1$-semifields}\ \\

\bigbreak

As mentioned in the overview, we characterize, step by step, the building blocks of the theory: terms, atomic formulas and general formulas.\\

\begin{flushleft}We first describe the terms of the languages $DL$ and $DL(L)$ .    \end{flushleft}

\begin{defn}
The set of $\mathcal{L}$-\emph{terms} is the smallest set $\mathcal{T}$ such that
\begin{enumerate}
  \item $c \in \mathcal{T}$ for each constant symbol $c \in \mathcal{C}$.
  \item Each variable symbol $v_i \in \mathcal{T}$ for $i=1,2,...$ \ .
  \item If $t_1,...,t_{n(f)} \in \mathcal{T}$ and $f \in \mathcal{F}$, then $f(t_1,...,t_{n(f)}) \in \mathcal{T}$.
\end{enumerate}
\end{defn}

The language $\mathcal{L}^{\star}$ contains the binary function symbols $\cdot$ and $+$,
taking $x_{1},...,x_{n}$ to be variable symbols. Denote $x^{k} = x \cdot \dots \cdot x$ taken $k \in \mathbb{N}$ times and $sx = x  + \ \cdots \ + x$ taken $s \in \mathbb{N}$ times.
Now, since $\pi_{1} \circ \pi_{1} = \pi_{1}$, $\pi_{2} \circ \pi_{2} =\pi_{2}$ and $\pi_{1} \circ \pi_{2} = \pi_{2} \circ \pi_{1} = 1$,
a general $\mathcal{L}^{\star}$-term can be formally written as
$$t(x_1,...,x_n,\pi_{1}(x_1),...,\pi_{1}(x_n),\pi_{2}(x_1),...,\pi_{2}(x_n))$$
$$ = p(x_1,...,x_n,\pi_{1}(x_1),...,\pi_{1}(x_n),\pi_{2}(x_1),...,\pi_{2}(x_n))$$
where $p \in \mathbb{N}[x_1,...,x_n,\pi_{1}(x_1),...,\pi_{1}(x_n),\pi_{2}(x_1),...,\pi_{2}(x_n)]$ is a polynomial.\\

The language $\mathcal{L}^{\star}_{L}$ contains the additional constant symbols $\{ \ell^{\prime} \ : \ \ell^{\prime} \in L^{\prime} \}$, which implies that  a general $\mathcal{L}^{\star}_{L}$-term can be written as
$$t(x_1,...,x_n,\pi_{1}(x_1),...,\pi_{1}(x_n),\pi_{2}(x_1),...,\pi_{2}(x_n), \ell^{\prime}_1,...,\ell^{\prime}_k)$$
$$ = p(x_1,...,x_n,\pi_{1}(x_1),...,\pi_{1}(x_n),\pi_{2}(x_1),...,\pi_{2}(x_n))$$
where $p \in (\mathbb{N} \cdot L^{\prime})[x_1,...,x_n,\pi_{1}(x_1),...,\pi_{1}(x_n),\pi_{2}(x_1),...,\pi_{2}(x_n)]$ is a polynomial. Here $\mathbb{N} \cdot L^{\prime} \doteq \{ k \ell^{\prime} \ : \ k \in \mathbb{N}, \ \ell^{\prime} \in L^{\prime} \}$. This merely translates to assigning $L$-layers to the coefficients of the polynomial.  Note that endowed with the algebraic structure of $L$ (by the sentences in \eqref{restriction_axioms1}), $L^{\prime}$ is closed under $+$ and $\cdot$.

Note that $\mathbb{N}$ is used in a formal manner, before interpretation of the language takes place.

\bigbreak

We proceed to describing atomic $\mathcal{L}^{\star}$ - formulas.
\begin{defn}
$\phi$ is an atomic $\mathcal{L}$-formula if either
\begin{enumerate}
  \item $t_1=t_2$ where $t_1$ and $t_2$ are terms.
  \item $R(t_1,...,t_{n(R)})$ where $R$ is $n(R)$-ary  relation of $\mathcal{L}$ and $t_1,...,t_{n(R)}$ are terms.
\end{enumerate}
\end{defn}

In the language $\mathcal{L}^{\star}$, there is only one relation symbol (besides $=$), namely, $<$.
Thus by the definition the atomic formulas in our language are
\begin{enumerate}
  \item $t_1=t_2$
  \item $t_1<t_2$
\end{enumerate}
where $t_1$ and $t_2$ are terms.\\

Although the language $\mathcal{L}^{\star}_{L}$ contains the additional unary relation symbol $P_{L}$,
since $DL(L) \models (x = \pi_2(x)) \leftrightarrow P_{L}(x)$, the atomic formulas of the form $P(t)$ are equivalent to the atomic formulas $t=\pi_2(t)$ and can be omitted from our discussion.

\begin{defn}
The set of $\mathcal{L}$-formulas is the smallest set $\Omega$ containing the atomic formulas
such that
\begin{enumerate}
  \item if $\phi \in \Omega$, then $\neg \phi \in \Omega$;
  \item if $\phi, \psi \in \Omega$, then $\phi \wedge \psi , \ \phi \vee \psi \in \Omega$;
  \item if $\phi \in \Omega$, then $\exists v_i \phi$ and $\forall v_i \phi$ are in $\Omega$, where $v_i$ is a subset of the variables in $\phi$.
\end{enumerate}

\end{defn}

\begin{flushleft} By all theories defined above, we have that \end{flushleft}
$$\forall x \forall y \ (x < y ) \leftrightarrow (\pi_{1}(x) < \pi_{1}(y))$$ and
$$\forall x \forall y \ ( x = y  \leftrightarrow (\pi_{1}(x)=\pi_{1}(y)) \wedge (\pi_{2}(x)=\pi_{2}(y)).$$
We can replace the atomic formulas $t_1=t_2$ by $(\pi_{1}(t_1) = \pi_{1}(t_2)) \wedge (\pi_{2}(t_1) =\pi_{2}(t_2))$
and $t_1 < t_2$ by $(\pi_{1}(t_1) < \pi_{1}(t_2))$.\\

\begin{rem}
Although $\pi_1(x)$ and $\pi_2(x)$ are restrictions of the variable $x$ to a specified subset of elements in the universe, there is no  loss of generality referring to them as general variables. Thus, for simplicity of notation, we omit specifying these restrictions in the current discussion.
\end{rem}
\bigbreak

As we have shown, a term $t$ of $\mathcal{L}^{\star}$ is just a polynomial in $\mathbb{N}[x_1,...,x_n]$ or a polynomial in $(\mathbb{N} \cdot L^{\prime})[x_1,...,x_n]$ for $\mathcal{L}^{\star}_{L}$, where $x_1,...,x_n$ are the set
of variables occurring in $t$. In the following discussion and throughout the rest of this section, there is no need to discriminate between these cases, as the assertions made apply to both of them. The only modification needed to adjust the statements to $\mathcal{L}^{\star}_{L}$ (instead of $\mathcal{L}^{\star}$), are declarative.

\bigbreak

\begin{defn}
Let $t$ be a term. Write $t(x_1,...,x_n) = \sum_{i=0}^{m} t_{i}(x_1,...,x_n)$ where $t_{i}$ are monomial terms for $i = 0,...,m$.
We define a \emph{monomial-terms-ordering (MTO)}, \newline $O(t_0,...,t_m)$, of $t$ to be a partition of $\{0,...,m \}$ into
two distinct sets $I,J$ such that for any $i,j \in I$, $\pi_1(t_i) = \pi_1(t_j)$, and for any
$k \in J$, $\pi_1(t_k) < \pi_1(t_j)$ for any $j \in I$. We also denote an MTO of $t$ by $O(t)$.
\end{defn}

\begin{note}
MTO is defined to distinguish the dominant (essential) monomials from all other non-essential monomials comprising a term.
\end{note}

\bigbreak

Let $t= \sum_{i=0}^{m} t_{i}$ be a term. Given an MTO ordering $O(t_0,...,t_m)$ of $t$,
we can use the sentences in \eqref{sentence} (all four axioms in the paragraph), along with
$\pi_{1} \circ \pi_{1} = \pi_{1}$, $\pi_{2} \circ \pi_{2} =\pi_{2}$ and
$\pi_{1} \circ \pi_{2} = \pi_{2} \circ \pi_{1} = 1$, to rewrite $t$ in
equivalent form as follows:
\begin{align*}
t(x_1,...,x_n) = & \pi_1(t(x_1,...,x_n))\pi_2(t(x_1,...,x_n)) \\
& = t'_{1}(\pi_{1}(x_1),...,\pi_{1}(x_n))t'_{2}(\pi_{2}(x_1),...,\pi_{2}(x_n))\end{align*}
for appropriate terms $t'_{1}$ and $t'_{2}$ satisfying $\pi_{i}(t(x_1,...,x_n)) = t'_{i}(\pi_{i}(x_1),...,\pi_{i}(x_n))$.

\bigbreak

Explicitly, consider the expression
$$\bigvee_{O(t_0,...,t_m)}\bigg(\bigwedge_{\substack{ i,j \in I \\ i > j }}\big(\pi_1(t_i) = \pi_1(t_j)\big) \wedge \bigwedge_{\substack{i \in I \\ j \in J} }\big(\pi_1(t_i) > \pi_1(t_j)\big)\bigg)$$
where the disjunction is taken over all possible MTO of $t$ and
$I,J \subseteq \{0,...,m \}$ is the partition defined by $O(t_0,...,t_m)$.
This expression is a tautology, as for any evaluation of the variables in $t$, at least one of the terms in the
disjunct must hold.
For each disjunct, by the considerations given above, we have that

$$\bigwedge_{\substack{ i,j \in I \\ i > j }}(\pi_1(t_i) = \pi_1(t_j))
\wedge \bigwedge_{\substack{ i \in I \\ j \in J } }(\pi_1(t_i) > \pi_1(t_j))
\rightarrow t = ( \pi_{1}(t_i) (\sum_{j \in J} \pi_{2}(t_j)))$$

\bigbreak

\begin{rem}
Note that if there exists a monomial term $t_{i_0}$ such that
$$\bigwedge_{\substack{ i,j \in I \\ i > j }}(\pi_1(t_i) = \pi_1(t_j))
\wedge \bigwedge_{\substack{ i \in I \\ j \in J } }(\pi_1(t_i) > \pi_1(t_j))$$
is false for every evaluation for any ordering of $t$ such that $t_{i_0} \in I$, then $t_{i_0}$ will not affect
the set of equivalent expressions for $t$, and thus can be omitted from $t$. After all
such terms are omitted, $t$ is reduced to what is known as its \emph{essential form}.
\end{rem}

In what follows, for simplicity of notation, we denote
$$\Delta(t_{0},...,t_{m}) \doteq \bigwedge_{\substack{ i,j \in I \\ i > j }}(\pi_1(t_i) = \pi_1(t_j))
\wedge \bigwedge_{\substack{ i \in I \\ j \in J } }(\pi_1(t_i) > \pi_1(t_j))$$
for the ordering $O(t_{1,0},...,t_{1,m_1}$ corresponding to partition $(I,J)$ of $\{0,...,m \}$.

Write $t_1(x_1,...,x_n) = \sum_{i=0}^{m_1} t_{1,i}(x_1,...,x_n)$ and $t_2(x_1,...,x_n) = \sum_{j=0}^{m_2} t_{1,j}(x_1,...,x_n)$
where $t_{1,i},t_{2,j}$ are monomial terms for $i = 0,...,m_1$ and $j = 0,...,m_2$.
By the above observations we have that

\medbreak

\small
\begin{align*} t_1 = t_2 & & & \leftrightarrow \bigwedge_{O(t_1),O(t_2)} \bigg(\Big(\Delta(t_{1,0},...,t_{1,m_1}) \wedge \Delta(t_{2,0},...,t_{2,m_2})\Big)
\rightarrow & & & & & & & & & \end{align*}
\begin{align*} & & & &\Big(\pi_{1}(t_{1,i_1}) \Big(\sum_{j \in I_1} \pi_{2}(t_{1,j})\Big)= \pi_{1}(t_{2,i_2}) \Big(\sum_{j \in I_2} \pi_{2}(t_{2,j})\Big)\Big)\bigg)  \end{align*}
\begin{align*} & & & & & & \leftrightarrow \bigwedge_{O(t_1),O(t_2)} \bigg(\Big(\Delta(t_{1,0},...,t_{1,m_1}) \wedge \Delta(t_{2,0},...,t_{2,m_2})\Big) \rightarrow  & & & & & & & & & \end{align*}
\begin{align*} & & & & \big(\pi_{1}(t_{1,i_1}) = \pi_{1}(t_{2,i_2})\big) \wedge \Big(\Big(\sum_{j \in I_1} \pi_{2}(t_{1,j})\Big)= \Big(\sum_{j \in I_2} \pi_{2}(t_{2,j})\Big)\Big)\bigg)\end{align*}

\begin{flushleft}and \end{flushleft}

\small
\begin{align*} t_1 < t_2  & & & \leftrightarrow \bigwedge_{O(t_1),O(t_2)} \bigg(\Big(\Delta(t_{1,0},...,t_{1,m_1}) \wedge \Delta(t_{2,0},...,t_{2,m_2})\Big)
\rightarrow & & & & & & & & & \end{align*}
\begin{align*} & & & & \Big(\pi_{1}(t_{1,i_1}) \Big(\sum_{j \in I_1} \pi_{2}(t_{1,j})\Big) <  \pi_{1}(t_{2,i_2}) \Big(\sum_{j \in I_2} \pi_{2}(t_{2,j})\Big)\Big)\bigg)  \end{align*}
\begin{align*} & & & & & & \leftrightarrow \bigwedge_{O(t_1),O(t_2)} \bigg(\Big(\Delta(t_{1,0},...,t_{1,m_1}) \wedge \Delta(t_{2,0},...,t_{2,m_2})\Big) \rightarrow \Big(\pi_{1}(t_{1,i_1}) < \pi_{1}(t_{2,i_2})\Big)\bigg).\end{align*}

Here $\bigwedge_{O(t_1),O(t_2)}$ runs over all possible distinct ordering of $t_1$ and $t_2$, where $I_1,I_2$ are \linebreak determined by the orderings and $i_1 \in I_1, i_2 \in I_2$ can be taken to be an element from each of the sets.

\medbreak

\begin{rem}\label{rem1}
By the theory of first order logic, the above expression can be expressed in \linebreak disjunctive normal form. In particular,
the logical expression $A \rightarrow B$ is equivalent to $\neg A \vee B$.
\end{rem}

\bigbreak

Using this last observation, we now characterize a general quantifier free formula in the language $\mathcal{L}^{\star}$.
Let $\psi$ be a quantifier free, $\mathcal{L}^{\star}$-formula. Then $\psi$ can be written in disjunctive normal form.
Namely, using the above assertions, there are atomic or negated (expressed by the $<$ relation) atomic formulas $\theta_{i,j}(\bar(x))$
of the following forms:
\medbreak

$$1. \ p(\pi_{1}(x_1),...,\pi_{1}(x_n)) = q(\pi_{1}(x_1),...,\pi_{1}(x_n))$$
$$2. \ p(\pi_{1}(x_1),...,\pi_{1}(x_n)) < q(\pi_{1}(x_1),...,\pi_{1}(x_n))$$

 \medbreak

\begin{flushleft} where $p,q$ are monomials. Here we note that the relations $\neq, \geq, \leq$ can all be expressed in the form of conjunctions, and disjunctions of the relations $=$ and $<$.\end{flushleft}
By the above assertions we have that
$$p(\pi_{1}(x_1),...,\pi_{1}(x_n))=\pi_1(p(x_1,...,x_n))$$ and  $$q(\pi_{1}(x_1),...,\pi_{1}(x_n))=\pi_1(q(x_1,...,x_n)),$$
thus these expressions can be inverted to obtain the monomial equations

$$ Q_1(\bar{x}) \doteq p(\pi_{1}(x_1),...,\pi_{1}(x_n)) \cdot q(\pi_{1}(x_1),...,\pi_{1}(x_n))^{-1} = 1$$
$$ P_1(\bar{x}) \doteq p(\pi_{1}(x_1),...,\pi_{1}(x_n)) \cdot q(\pi_{1}(x_1),...,\pi_{1}(x_n))^{-1} < 1$$

\begin{flushleft} and of the form \end{flushleft}
$$3. \ r(\pi_{2}(x_1),...,\pi_{2}(x_n)) = s(\pi_{2}(x_1),...,\pi_{2}(x_n))$$
$$4. \ r(\pi_{2}(x_1),...,\pi_{2}(x_n)) \neq s(\pi_{2}(x_1),...,\pi_{2}(x_n))$$

\begin{flushleft}where $r,s \in \mathbb{N}[\lambda_1,...,\lambda_n]$ are polynomials over $\mathbb{N}$ (we use $\lambda_i$ here to denote variables in order to avoid confusion with the $x_i$'s).\end{flushleft}
\begin{flushleft}Denote \end{flushleft}
$$ Q_2(\bar{x}) \doteq  r(\pi_{2}(x_1),...,\pi_{2}(x_n)) = s(\pi_{2}(x_1),...,\pi_{2}(x_n))$$
$$ P_2(\bar{x}) \doteq  r(\pi_{2}(x_1),...,\pi_{2}(x_n)) \neq s(\pi_{2}(x_1),...,\pi_{2}(x_n))$$

\begin{flushleft}so that\end{flushleft}
$$ \mathrm{T} \models \psi(\bar{x}) \leftrightarrow \bigvee_{i=1}^{n} \bigwedge_{j=1}^{m}\theta_{i,j}(\bar{x}).$$
%

\bigbreak

\begin{rem}
Let $G$ be a layered divisibly closed 1-semifield. Then, by definition, we must have $1,0,c \in G$ for some element $c$ such that $\pi_1(c) \not \in \{0,1\}$ which implies that $\pi_1(c) > 1$. Since every element of $x \in G$ is of the form $\pi_1(x) \cdot \pi_2(x)$ and vice versa, it is sufficient to consider $\pi_1(G)$ and $\pi_2(G)$. Now, $\pi_1(c) \in \pi_1(G)$ so $\pi_1(G) \setminus \{ 0 \}$ is not the trivial group. Thus, as $\pi_1(G) \setminus \{ 0 \}$ is divisibly closed we have that $\mathbb{Q}_{\geq 0} \subseteq \pi_1(G)$ is embeddable in $\pi_1(G)$. Finally, as  $\{ 1 \} \subseteq \pi_2(G)$, we have that
\begin{equation} \label{Theta} \Theta = 1 \times \mathbb{Q}_{ \geq 0}\end{equation}
is embeddable in $G$. Moreover, as $\Theta \models DLSF$ we have that for any $\mathcal{L}^{\star}$-structure $\mathcal{M}$ such that $\mathcal{M} \models DLSF$, $\Theta$ embeds into $\mathcal{M}$.\\
Analogously, for a given semiring$^{\dag}$ \ $L$, define
\begin{equation} \label{Theta_L} \Theta_L = L \times \mathbb{Q}_{ \geq 0} .\end{equation}
Then $\Theta_L$ embeds in every $\mathcal{L}^{\star}_{L}$-structure $\mathcal{M}$ such that $\mathcal{M} \models DLSF(L)$.
\end{rem}

\begin{cor}\label{cor1}
$\Theta$ embeds into every model of  \textbf{DLSF} and  for a given semiring$^{\dag}$  $L$, \linebreak $\Theta_L$ embeds into every model of  \textbf{DLSF(L)}.
\end{cor}

\begin{rem}
$$\Theta \cong ( \{ 1 \} \times \mathbb{Q}_{+} , +, \cdot, \pi_{1}, \pi_{2}, <, (0,1),(1,1) )$$
where $\pi_{1}(x,y) = x$, $\pi_{2}(x,y) = y$,
$$(x_1,y_1) < (x_2,y_2) \Leftrightarrow (x_1 < x_2) \vee ((x_1 = x_2) \wedge (y_1 < y_2)),$$
$$(x_1,y_1) \cdot (x_2,y_2) = (x_1x_2,y_1y_2)$$
and
$$(x_1,y_1) + (x_2,y_2) =
  \begin{cases}
   (x_1,y_1) &  \ x_1 > x_2, \\
   (x_2,y_2) &  \ x_2 > x_1, \\
   (x_1, y_1 + y_2) & \ x_1 = x_2.
   \end{cases}$$

\end{rem}

\begin{defn}
Let $\mathrm{T}$ be an $\mathcal{L}$-theory. $\mathrm{T}_{\forall}$ is the set of all universal
consequences of $\mathrm{T}$.
\end{defn}

\begin{rem}
Let $\mathrm{T}$ be an $\mathcal{L}$-theory. For an $\mathcal{L}$-structure $\mathcal{A}$,
$\mathcal{A} \models \mathrm{T}_{\forall}$ iff there exists an \linebreak $\mathcal{L}$-structure $\mathcal{M}$ such that $\mathcal{M} \models \mathrm{T}$  and $\mathcal{A} \subseteq \mathcal{M}$ ($\mathcal{A}$ is a substructure of $\mathcal{M}$).
\end{rem}

\newpage
\section{The completeness of the theory of $L$-layered divisibly \\ \ \ \ \ closed $1$-semifields} \ \\

We now proceed to prove that the theory of $L$-layered divisibly closed $1$-semifields is a complete theory.

\begin{flushleft}The following lemma is proved in \cite{Layered}:\end{flushleft}

\begin{lem}
Let $G$ be a uniform $L$-layered domain. Then there exists a divisibly closed, \linebreak uniform $L$-layered
$1$-semifield $G'$ and an embedding $i : G \rightarrow G'$ such that if
$j: G \rightarrow H$ is an embedding of $G$ into a divisibly closed uniform $L$-layered semifield $H$,
then there exists a map $h : G' \rightarrow H$ such that $j = h \circ i$.
\end{lem}

\begin{proof}
Take the $G'$ to be the  divisible closure of the 1-layered semifield of fraction of $G$ described
in \cite{Layered}.
\end{proof}

\bigbreak

\begin{prop}\label{universal1}
$DLSF_{\forall}$ is the theory of uniform layered semidomains, $DL$ .
\end{prop}

\begin{proof}
As all sentences in the theory of layered domains, $DL$, are universal, $DL$ is contained in $DLSF(L)_{\forall}$
(since every divisibly closed layered semifield is particularly a layered semidomain).
On the other hand, by the above lemma and the model theoretic remark, we have that $DLSF(L)_{\forall}$ is contained
in the theory of $L$-layered semidomains, namely, since every structure that models the theory of layered semidomains  can be embedded in a model of $DLSF$ with the same layering semiring$^{\dag}$ it admits all universal consequences of $DLSF$, i.e., it also models $DLSF_{\forall}$.
\end{proof}

\begin{prop}
$DLSF(L)_{\forall}$ is just the theory of uniform $L$-layered semidomains, $DL(L)$.
\end{prop}
\begin{proof}
The derivation of $\mathcal{L}^{\star}_{L}$-theory $DLSF(L)$ from the $\mathcal{L}^{\star}$-theory $DLSF$ was done introducing only universal sentences, leaving all sentences of the $\mathcal{L}^{\star}_{L}$- theory of $L$-layered domains universal, and thus $DL(L)$ it is contained in $DLSF(L)_{\forall}$. The rest of the proof is as in the proof of Proposition \ref{universal1}.
\end{proof}

\begin{defn}
Let $\mathrm{T}$ be an $\mathcal{L}$-theory. We say that $\mathrm{T}$ has \emph{algebraically-prime-models} if for any
$\mathcal{A} \models \mathrm{T}_{\forall}$ there is $\mathcal{M} \models \mathrm{T}$ and an embedding
$i: \mathcal{A} \rightarrow \mathcal{M}$ such that for each $\mathcal{N} \models \mathrm{T}$ and embedding
$i: \mathcal{A} \rightarrow \mathcal{N}$ there is a map $h: \mathcal{M} \rightarrow \mathcal{N}$ such that
$j = h \circ i$.
\end{defn}

\begin{rem}
Since $DLSF_{\forall}$ is the theory of layered semidomains, using the above lemma again
yields  that $DLSF$ has algebraically-prime-models. In the same manner, $DLSF(L)$ has algebraically-prime-models.
\end{rem}

\begin{flushleft}The following result is proved in \cite{Mark}: \end{flushleft}
\begin{prop}
Let $\mathrm{T}$ be an $\mathcal{L}$-theory satisfying the following two conditions:
\begin{enumerate}
  \item $\mathcal{T}$ has algebraically-prime-models.
  \item For $\mathcal{M},\mathcal{N} \models \mathrm{T}$ such that $M \subseteq N$, and
  for any quantifier free formula $\psi(\bar{v},w)$, and any $\bar{a} \in M$,
  if $N \models \exists w \psi(\bar{a},w)$ then $M \models \exists w \psi(\bar{a},w)$.
 Then there exists $c \in G$ such that $G \models \psi(\bar{a},c)$.
\end{enumerate}
Then $\mathrm{T}$ has quantifier elimination.
\end{prop}

\bigbreak
\bigbreak

We have already shown that $DLSF$ and $DLSF(L)$ satisfy the first condition. As for the second condition, it is only attained by $DLSF(L)$, as we will show next.

In the following proof, we make use of a technique introduced in \cite{Mark} for proving
that the theory of ordered divisible groups has quantifier elimination.

\begin{prop}\label{main}
Let $G$ and $H$ be layered divisibly closed semifields,
$G \subset H$, and $\psi(\bar{v},w)$ is a quantifier free formula.
Let $\bar{a} \in G$ and $b \in H$ such that $H \models \psi(\bar{a},b)$.
 Then there exists $c \in G$ such that $G \models \psi(\bar{a},c)$.
\end{prop}

\begin{proof}
First note, as remarked above, that $\psi(\bar{v},w)$ can be put in disjunctive normal form $\bigvee_{i=1}^{n} \bigwedge_{j=1}^{m}\theta_{i,j}(\bar(v),w)$  where $\theta_{i,j}$ is of one of the forms $Q_1, P_1, Q_2$ and $P_2$.
\\Because $H \models \psi(\bar{a},b)$ we have that $H \models \bigwedge_{j=1}^{m}\theta_{i,j}(\bar(a),b)$ for some $i$.
Thus, following \linebreak Remark~\ref{rem1}, we may assume $\psi$ is a conjunction of atomic and negated atomic formulas of the above forms $Q_1, P_1, Q_2$ and $P_2$,
which, in turn,  are just atomic formulas of the forms $Q_1, P_1, Q_2$ and $P_2$.
As the monomials of the form $Q_1$ and $P_1$ do not contain the $+$ binary function, we let ourselves pass to logarithmic notation where $+$ replaces $\cdot$ and $0$ replaces $1$. So, an atomic formula $\theta(\bar{v},w)$ is equivalent to one of the following forms:
$$\sum n_i \pi_{1}(v_i) + m \pi_{1}(w) = 0$$
$$\sum n_i \pi_{1}(v_i) + m \pi_{1}(w) > 0$$
$$r( \pi_{2}(v_1),...,\pi_2(v_{n-1}),\pi_{2}(w)) = s( \pi_{2}(v_1),...,\pi_2(v_{n-1}),\pi_{2}(w))$$
$$r( \pi_{2}(v_1),...,\pi_2(v_{n-1}),\pi_{2}(w)) \neq s( \pi_{2}(v_1),...,\pi_2(v_{n-1}),\pi_{2}(w))$$
\bigbreak

\begin{flushleft}where $r(x_1,....,x_{n-1},y), s(x_1,....,x_{n-1},y) \in \mathbb{N}[x_1,...,x_{n-1},y]$.\end{flushleft}
Note first, that for a polynomial \ $p(x_1,....,x_{n-1},y) \in (\mathbb{N} \cdot L^{\prime})[x_1,...,x_{n-1},y]$, one can consider $p(\pi_2(a_1),....,\pi_2(a_{n-1}),y)$ as a polynomial in $\pi_2(G)[y]$ (remember that $L^{\prime} \subset Im(\pi_2)$). \linebreak
Second, note that there is an element $g \in G$ such that $g = -\sum n_i \pi_{1}(a_i).$
(Here $-g$ is just $1 \div g$).\\
Now, as $\pi_{1}(g) = \pi_{1}(-\sum n_i \pi_{1}(a_i)) = -\sum n_i \pi_{1}(\pi_{1}(a_i)) = -\sum n_i \pi_{1}(a_i) = g$ we can replace the above forms by the following:
$$ m w_1 = g$$
$$ m w_1 > g$$
$$r(w_2) = s(w_2)$$
$$r(w_2) \neq s(w_2)$$

\begin{flushleft}where $w_1 = \pi_{1}(w), w_2 = \pi_2(w)$, $g = \pi_{1}(g)$ and $r,s \in \pi_2(G)[y]$.\end{flushleft}
Thus we may assume that
$$\psi(\bar{a},w) \leftrightarrow \bigwedge( m_iw_1 = g_{i}) \wedge \bigwedge( r_j(w_2) = s_j(w_2) )
\wedge \bigwedge( n_iw_1 < h_{i}) \wedge \bigwedge( p_{j}(w_2) \neq q_j(w_2))$$  where $g_{i}, h_{i} \in G$
such that $g_{i} = \pi_{1}(g_{i})$,  $h_{i} = \pi_{1}(h_{i})$, $r_j,s_j,p_j,q_j \in \pi_2(G)[y]$ and $m_i, n_i \in \mathbb{Z}$.\\

If there is actually a conjunct $m_i w_1 = g_{i}$, then we must have
$b = \frac{g_{j,i}}{m_i}$, and since $\pi_{1}(G)$ is divisibly closed,
we have that $b \in \pi_{1}(G)$. Thus $b \in G$ and we take $c=b$ completing the proof.
If there exists an index $j$ such that one of the polynomials $r_j$ and $s_j$ is nonzero, then $r_j(b) = s_j(b)$ with $b = \pi_2(b)$ (i.e., $b \in \pi_2(G)$).
Now, by axioms in  \eqref{restriction_axioms1}, \eqref{restriction_axioms2}, \eqref{restriction_axioms3}  we have an element $c = \pi_2(c) \in G$ (i.e. $c \in \pi_2(G)$) corresponding to $b$, such that $r_j(c) = s_j(c)$, again, completing the proof. If both cases are not attained, then
\begin{equation}\label{eq1}
\psi(\bar{a},w) \leftrightarrow \bigwedge( n_iw_1 > h_{i}) \wedge \bigwedge(p_{j}(w_2) \neq q_j(w_2)).
\end{equation}
Let $k_{0} = \min \{ \frac{h_{i}}{n_i} \ : \ n_i < 0 \}$ , $k_{1} = \max \{ \frac{h_{i}}{n_i} \ : \ n_i > 0 \}$.
Then $b \in H$ satisfies $\psi(\bar{a},w)$ if and only if $k_{0} < \pi_{1}(b) < k_{1}$. Because $b$ satisfies $\psi$, we must have  $k_{0} < k_{1}$. As $\pi_{1}(G)$ is divisibly closed, it is densely ordered since for
$s,t \in \pi_{1}(G)$ such that $s < t$, we have that $\frac{s+t}{2} \in \pi_{1}(G)$ and $s < \frac{s+t}{2} < t$
(remember we use logarithmic notation). So there is $d \in \pi_{1}(G)$ such that $k_{0} < d < k_{1}$.
Now, from the same reason given for the case of equality, the existence of $b \in \pi_2(H)$ such that $p_{j}(w_2) \neq q_j(w_2)$ \ for all indices $j$, implies the existence of an element $e \in \pi_2(G)$ such that $p_{j}(e) \neq q_j(e)$ \ for all $j$.
Taking $c= d \cdot e \in G$, we have that $\pi_1(c) = \pi_1(d) \cdot \pi_1(e) = d \cdot 1 = d$ and $\pi_2(c) = \pi_2(d) \cdot \pi_2(e) = 1 \cdot e = e$. So $c \in G$ satisfies all inequalities of \eqref{eq1}, thus completing our proof.
\end{proof}


\begin{flushleft}This last proposition proves that\end{flushleft}
\begin{prop}\label{qel}
$DLSF(L)$ has quantifier elimination.
\end{prop}

%
%
%
%

\begin{defn}
An $\mathcal{L}$-theory is \emph{model-complete} if $\mathcal{M} \prec \mathcal{N}$
whenever $\mathcal{M} \subset \mathcal{N}$ and $\mathcal{M}, \mathcal{N}~\models~\mathrm{T}$.\\
\end{defn}

\medbreak

Other well-known results (\cite{Mark}) in model theory are given in the following:

\begin{prop}
If $\mathrm{T}$ has quantifier elimination, then $\mathrm{T}$ is model-complete.
\end{prop}

\begin{prop}\label{prop1}
Let $\mathrm{T}$ be a model-complete theory. Suppose that there is $\mathcal{M}_{0} \models \mathrm{T}$ such that $\mathcal{M}_{0}$ embeds into every model of $\mathrm{T}$. Then $\mathrm{T}$ is complete.
\end{prop}

\begin{cor}
$DLSF(L)$ is a complete theory.
\end{cor}

\begin{proof}
Since $\Theta_L$ (eq. \ref{Theta}) embeds into every model of $DLSF(L)$, Proposition \ref{prop1} yields
that $DLSF(L)$ is a complete theory.
\end{proof}

One important consequence of the theory introduced above is the following:

\begin{defn}
Let $(F,+,\cdot,0,1)$ be a linearly ordered field. The max-plus algebra $MP(F)$
is a semifield $(\bar{F}= F \cup -\infty,\oplus,\odot,-\infty,0)$  such that
for $a,b \in \bar{F}$
\begin{align*} a \oplus b = \max(a,b), \ \ \  a \odot b = a + b. \end{align*}
Note that $-\infty$ serves as the identity element with respect to the $\oplus$ operation
while $0$ serves as the identity element with respect to the $\odot$ operation.
\end{defn}

The above definition is a generalization of the traditional definition of the max-plus algebra
where $F$ is taken to be the real number field $\mathbb{R}$.

\begin{prop}
The theory of divisibly closed max-plus algebras is complete.
\end{prop}

\begin{proof}
A max-plus algebra can be defined simply by taking the layering semiring$^{\dag}$ \ $L$ to be the idempotent semiring$^{\dag}$  $\{1 \}$ (idempotent in the sense that $1+1=1$).
conversely, taking the layering semiring$^{\dag}$ \ $L$ to be idempotent semiring$^{\dag}$ $\{ 1 \}$ yields a max-plus algebra.
Taking $L$ to be $\{ 1 \}$ is equivalent to taking the unary function $\pi_2$ to admit $\forall x \ \pi_2(x)=1$.
Thus adding the sentence $\forall x \ \pi_2(x)=1$ to the theory of LD yields the theory of max-plus algebras, and
adding it to the theory of DLSF yields the theory of divisibly closed max-plus algebras.
This implies that the theory of divisibly closed max-plus algebras has quantifier elimination.
Moreover, we have
$$\Theta = 1 \times \mathbb{Q}_{ \geq 0} \cong \mathbb{Q}_{ \geq 0}$$ which embeds into
every divisibly closed max-plus algebra, which is thus itself a complete theory.


\end{proof}

\bibliographystyle{amsplain}
\bibliography{TAGBib}

\providecommand{\bysame}{\leavevmode\hbox to3em{\hrulefill}\thinspace}
\providecommand{\MR}{\relax\ifhmode\unskip\space\fi MR }
\providecommand{\MRhref}[2]{%
  \href{http://www.ams.org/mathscinet-getitem?mr=#1}{#2}
}
\providecommand{\href}[2]{#2}
\begin{thebibliography}{1}

\bibitem{Layered}
Z.~Izhakian, M.~Knebusch, and L.~Rowen, \emph{Layered tropical mathematics},
  preprint (2013), available at \texttt{http://arxiv.org/pdf/0912.1398.pdf}.

\bibitem{Mark}
D.~Marker, \emph{Model {T}heory: {A}n {I}ntroduction}, Springer, 2002.

\end{thebibliography}

\end{document}